\documentclass[reqno]{amsart} \usepackage{amscd}
\usepackage{epsf}
\newtheorem{theorem}{Theorem}[section]
\newtheorem{proposition}[theorem]{Proposition}

\theoremstyle{remark} 
\newcommand{\field}[1]{\ensuremath{\mathbb{#1}}}
\newcommand{\CC}{\field{C}}

\begin{document}

\title[Geometric Quantization ]{ Geometric Quantization of the Hitchin System}

\author{Rukmini Dey}
\address{I.C.T.S., Bangalore\\
rukmini@icts.res.in; rukmini.dey@gmail.com}

\maketitle

\begin{abstract}
This paper is about geometric quantization of the Hitchin system. 
We quantize a Kahler  form on the Hitchin moduli space (which is half the first Kahler form defined by Hitchin) by considering  the  Quillen bundle as the prequantum line bundle and modifying  the Quillen metric using the Higgs field  so that the curvature is proportional to the Kahler form. We show that this Kahler form is 
integral and the Quillen bundle descends as a prequantum line bundle on the moduli space. It is holomorphic and hence one can take holomorphic square integrable sections as the Hilbert space of quantization of the Hitchin moduli space.
\end{abstract}

\section{{\bf Introduction}}

{\bf Geometric Quantization:}

Given a symplectic manifold $(M, \omega)$, with $\omega $ integral (i.e. its cohomology class is in $H^2(M, {\mathbb Z})$, 
geometric pre-quantization is the construction of a line bundle whose curvature $\rho$ is proportional to the symplectic form. This is always possible as long as $\omega$ is integral. The method of quantization,  developed by Kostant and Souriau,  assigns to functions $f \in C^{\infty}(M)$, an operator, $\hat{f} = -i \nabla_{X_f} +   f$ acting on the Hilbert space of square integrable sections of $L$. Here $\nabla =d - i \theta $  where $ \omega = d \theta $ locally and $X_f $ satisfies $\omega (X_f, \cdot) = -df (\cdot)$. This assignment has the property that  that the Poisson bracket (induced by  the symplectic form), namely, 
$\{ f_1, f_2 \}_{PB} $ corresponds  to an operator proportional to the commutator $[\hat{f}_1, \hat{f}_2]$ for any two functions $f_1 , f_2$.  

The Hilbert space obtained from the prequantization is too big for most purposes.
Geometric quantization involves  construction of a polarization of the symplectic manifold such that we now take polarized sections of the line bundle, yielding a finite dimensional Hilbert space in most cases. However, $\hat{f}$ does not 
map the polarized Hilbert space to the polarized Hilbert space  in general. Thus only a few observables from the set of all  $f \in C^{\infty}(M)$ are quantizable.

The general reference for this is Woodhouse, ~\cite{W}.

{\bf The Determinant line bundle of Quillen}

${\mathcal A} =$space of unitary connections on a vector bundle $E$
associated to
a principal $G$ bundle on a Riemann surface $\Sigma$. Can  identify 
${\mathcal A} = {\mathcal A}^{0,1}$ (since $A^{1,0*} = - A^{0,1}$) . 
On this infinite-dimensional space, Quillen defines  the determinant line bundle, ~\cite{Q}. 
One defines the Cauchy-Riemann operator on $\Sigma$ locally written as $d \bar{z} \frac{\partial}{\partial{\bar{z}}} + A^{0.1}$ which 
acts on sections of the vector bundle $E$.  We denote it by $\bar{\partial}_A$.
Construct a line bundle ${\mathcal L}$ on ${\mathcal A}^{0,1}$ as follows. 

The fiber on  $A^{0,1}$ is ${\rm det} \, \bar{\partial}_A = \wedge^{\rm{top}} ({\rm Ker} \, \bar{\partial}_A)^{*} \otimes \wedge^{\rm{top}}({\rm Coker} \, \bar{\partial}_A)$ 

The dimension of  ${\rm Ker} \, \bar{\partial}_A $ jumps even locally but the whole object makes sense --after certain identifications (Quillen's ingenious construction). Quillen further shows that  ${\mathcal L}$ carries a metric and a connection s.t. the K\"{a}hler potential is $ \frac{\partial \zeta_A}{\partial s} (0)$ where $\zeta_A$ is the $zeta$ function corresponding to the Laplacian of $\bar{\partial}_{A}$. The
 curvature turns out to be   proportional to 
$\Omega_Q(\alpha, \beta) =
- \frac{1}{2}\int_{\Sigma} {\rm Tr} (\alpha \wedge \beta), $ the K\"{a}hler form on ${\mathcal A}$. 
(the proportionality constant being $\frac{i}{2 \pi}$), ~\cite{Q}. We have shown the curvature computation in some detail in  ~\cite{D3}. 

{\bf The Hitchin system:}

In ~\cite{D} we described the Hitchin system  in a slightly unconventional  notation,  as follows.
Let $\Sigma$ be a compact Riemann 
surface of genus $g>1$ and let $P$ be a principal 
$U(n)$-bundle over $\Sigma$. Let $A$ be a unitary connection on $P$, 
i.e. $A = A^{(1,0)} + A^{(0,1)}$ such that 
$A^{(1,0)} = -A^{(0,1)*}$, where $*$ denotes conjugate transpose. Thus we can identify the space of all unitary 
connections with its $(0, 1)$-part, i.e. $A^{(0,1)}.$ 
Let $\Phi^{1,0}$ be a complex Higgs field, such that
$\Phi^{1,0} \in {\mathcal H} = \Omega^{1,0}(M; {\rm ad} P \otimes \CC)$. 

{\bf Note:} In ~\cite{H} this $\Phi^{1,0}$ is written as $\Phi$. But we will 
be using the present notation since we will need $\Phi^{0,1}$, defined as $\Phi^{(0,1)} = - \Phi^{(1,0)*}$ and our
$\Phi = \Phi^{1,0} + \Phi ^{0,1}.$ 
 
The pair $(A, \Phi^{1,0})$ will be said to satisfy the self-duality equations if

$(1)\rm{\;\;\;\;\;}F(A) = -[\Phi^{(1,0)}, \Phi^{(1,0)*}],$

$(2)\rm{\;\;\;\;\;}d^{\prime\prime}_A \Phi^{(1,0)} = 0.$

Here $F(A)$ is the curvature of the connection $A.$
The operator $d_A^{\prime \prime}$ is the $(0,1)$ part of the extension of the 
covariant derivative operator to act on $\Omega^{1,0}(M, {\rm ad} P \otimes \mathbb{C})$.  
There is a gauge group acting on the space of $(A, \Phi)$ which leave the 
equations invariant. If $g$ is an $U(n)$ gauge transformation then 
$(A_1, \Phi_1)$ and $(A_2, \Phi_2)$ are gauge equivalent if 
$d_{A_2} g = g d_{A_1}$ and $\Phi_2 g = g \Phi_1$ ~\cite{H}, page 69.
  Taking the quotient by the gauge group of the solution 
space to $(1)$ and $(2)$ gives  the moduli space of solutions to these 
equations and is denoted by ${\mathcal M}_H$.
Hitchin shows that there is a natural metric on the moduli space 
${\mathcal M}_H$ and further  proves that the metric is 
hyperK\"{a}hler ~\cite{H}.

Let the configuration space be defined as ${\mathcal C} = \{ (A^{0,1}, \Phi^{1,0})| 
A^{0,1} \in {\mathcal A}, \Phi^{1,0} \in {\mathcal H} \}$ where
${\mathcal A}$ is the space of unitary connections on $P$, identified with its $A^{(0,1)}$ part  
and 
${\mathcal H} = \Omega^{(1,0)}(\Sigma, {\rm ad} P \otimes \CC)$ is the space 
of Higgs fields. 
   
We can identify the space of unitary connections with its $(0,1)$ part 
and the tangent vector is also $(0,1)$ part of a $1$-form
Let $\alpha^{(0,1)}, \beta^{(0,1)} \in T_{A} {\mathcal A} = \Omega^{(0,1)} (\Sigma, {\rm ad} P\otimes \CC) $, such that $\alpha^{(1,0)} = - \alpha^{(0,1)*}$ and   $\beta^{(1,0)} = -\beta^{(0,1)*}$. 
Let
$\gamma^{(1,0)}, \delta^{(1,0)} \in T_{\Phi} {\mathcal H} = \Omega^{(1,0)} (\Sigma, \rm{ad} P \otimes \CC )$.
Let us extend $\gamma^{(1,0)}, \delta^{(1,0)}$ by defining 
$\gamma^{(0,1)} = -\gamma^{(1,0)*}$ and $\delta^{(0,1)} = - \delta^{(1,0)*}$. 
Thus, $\alpha = \alpha^{(0,1)} + \alpha^{(1,0)}$ and $ \beta = \beta^{(0,1)} +
\beta^{(1,0)}$, $\gamma = \gamma^{(0,1)} + \gamma ^{(1,0)}$ and $\delta = \delta^{(0,1)} + \delta^{(1,0)}$, i.e.  $\alpha, \beta, \gamma, \delta  \in \Omega^1(M, {\rm ad} P).$
Let $X, Y$ be two tangent vectors to the configuration space, given
 by $X= (\alpha^{(0,1)}, \gamma^{(1,0)})$, and $Y= (\beta^{(0,1)} , \delta^{(1,0)})$.

 In ~\cite{H}, (page 79 and page 88), Hitichin defines the metric on the moduli space
${\mathcal M}_H$ given as follows.
  
On $ T_{(A, \Phi)} {\mathcal C} = T_A {\mathcal A} \times T_{\Phi} {\mathcal H}$which is $\Omega^{(0,1)}(\Sigma, {\rm ad} P \otimes \CC) \times 
\Omega^{(1,0)}(\Sigma, {\rm ad} P \otimes \CC)$, Hitchin defines 
the metric $g_1$ such that (in our notation)
\begin{eqnarray*}
& & g_1((\alpha^{(0,1)},\gamma^{(1,0)}), (\alpha^{(0,1)}, \gamma^{(1,0)})) \\
&=& 2i  \int_{\Sigma} {\rm Tr} (\alpha^{(0,1)*} \wedge \alpha^{(0,1)})
+ 2i  \int_{\Sigma} {\rm Tr} (\gamma^{(1,0)} \wedge \gamma^{(1,0)*})
\end{eqnarray*}

As in ~\cite{D}, let us define a metric on the configuration space as
\begin{eqnarray*}
& &  g (X, Y) = g ((\alpha^{(0,1)}, \gamma^{(1,0)}), (\beta^{(0,1)} , \delta^{(1,0)}))  \\  &=&  -   \int_{\Sigma} {\rm Tr} (   \alpha \wedge *_1 \beta) - 2 {\rm Im} 
\int_{\Sigma} {\rm Tr}  ( \gamma^{(1,0)} \wedge *_2 \delta^{(1,0){\rm tr}}) 
\end{eqnarray*}
 Here Tr denotes  trace, * denotes conjgate tranpose,  $*_1$ denotes 
the Hodge star taking $dx$ forms to $dy$ forms and $dy$ forms to $-dx$ forms 
(i.e. $*_1 (\eta dz) = -i \eta dz$ and 
$*_1(\bar{\eta} d \bar{z}) = i \bar{\eta} d \bar{z}$)  and 
 $*_2$ denotes the  operation (another Hodge star), 
 such that $*_2(\eta dz) = \bar{\eta} d \bar{z}$ and 
$*_2(\bar{\eta} d \bar{z}) =  -\eta d z$. Here tr denotes transpose. See ~\cite{D} for details.

In ~\cite{D} we checked that this coincides with the  metric on the moduli space
${\mathcal M}_H$ given by Hitchin, namely $g = g_1$.

There is a complex structure, the descendent of the complex structure on the configuration space given by
$${\mathcal I} = \left[ \begin{array}{cc}
i   & 0 \\
 0  & i
\end{array} \right]. $$ 

We let 
$\Omega (X, Y)  = g(X, {\mathcal I}Y) = \int _{\Sigma}{\rm Tr} (\alpha \wedge \beta) - \int_{\Sigma} {\rm Tr} (\gamma \wedge \delta)$, see ~\cite{D} for details.

With notation as in ~\cite{D}, let $$\tilde{\Omega}(X, Y)  = \frac{1}{2} g({\mathcal I}X, Y) = - \frac{1}{2}   \int_{\Sigma} {\rm Tr} (\alpha \wedge \beta )+ \frac{1}{2} \int_{\Sigma} {\rm Tr} (\gamma \wedge \delta)$$ denote a Kahler form on the configuration space which descends to the moduli space ${\mathcal M}_H$. (This Kahler form  is a factor of $\frac{1}{2}$ times the first Kahler form on the Hitchin moduli space as defined in ~\cite{H}).

In this paper, we quantize the  Kahler form  $\tilde{\Omega}$ on the Hitchin moduli space ~\cite{H} by considering  the  Quillen bundle as the prequantum line bundle and modifying  the Quillen metric using the Higgs field  so that the curvature is proportional to the Kahler form.

Note that in ~\cite{D}, we had quantized  the first Kahler  form on the Hitchin moduli space  differently. There  one had to choose a gauge equivalent class of a fixed connection $A_0$, which seems unnatural.

\section{{\bf Prequantization and Quantization}}

Define a prequantum bundle ${\mathcal P} = \text{det} \bar{ \partial}_A $, the Quillen determinant line bundle, ~\cite{Q},  on the configuration space with the modified Quillen metric:
$exp ( - \zeta_A^{\prime}(0))$ modified by a factor $exp ( \frac{i}{4 \pi} \int_{\Sigma} \text{Tr}  ((\Phi^{1,0}  + \Phi^{0,1}))  \wedge \overline{(\Phi^{1,0} + \Phi^{0,1})}) )$.

The first term  in the metric contributes  $-\frac{i}{4 \pi}  \int_{\Sigma} \text{Tr} (\alpha \wedge \beta) $ to the curvature of ${\mathcal P}$,  ~\cite{Q},~\cite{D3}.
A simple calculation shows that the second term in the metric contributes 
\begin{eqnarray*}
\frac{i}{4 \pi} (\delta_{\Phi^{1,0}} \delta_{ \overline{\Phi^{1,0}}}  \int_{\Sigma}\text{Tr}  ((\Phi^{1,0} + \Phi^{0,1})  \wedge \overline{(\Phi^{1,0} + \Phi^{0,1}})) ) (X, Y) =\frac{i}{4 \pi} \int_{\Sigma}\text{Tr}  (\gamma \wedge \delta)
\end{eqnarray*}
to the curvature of ${\mathcal P}$.

In fact, in physicists' notation, 
\begin{eqnarray*}
& & \delta_ {\Phi^{1,0}} \delta_ {\overline{\Phi^{1,0}}} ( \int_{\Sigma}\text{Tr}  ((\Phi^{1,0} + \Phi^{0,1})  \wedge \overline{(\Phi^{1,0} + \Phi^{0,1}}))) \\
& =& \delta_ {\Phi^{1,0}} \delta_ {\Phi^{0,1}} ( \int_{\Sigma}\text{Tr}  ((\Phi^{1,0} + \Phi^{0,1})  \wedge (\Phi^{0,1} + \Phi^{1,0}))) \\
&=& \int_{\Sigma} \text{Tr} (\delta \Phi^{1,0} \wedge \delta \Phi^{0,1}) +  \int_{\Sigma} \text{Tr} (\delta \Phi^{0,1} \wedge \delta \Phi^{1,0})\\
&=&  \int_{\Sigma} \text{Tr} (\delta(\Phi^{1,0} + \Phi^{0,1}) \wedge \delta(\Phi^{0,1} + \Phi^{1,0}))
\end{eqnarray*}

Thus we have the following proposition: 
\begin{proposition}
The curvature of ${\mathcal P}$ is precisely $\frac{i}{2 \pi} \tilde{\Omega}$ on the configuration space.
\end{proposition}

\begin{proposition}
The descendant of this Kahler form on the  Hitchin moduli space, ${\mathcal M}_H$,  is in fact  integral.
\end{proposition}

\begin{proof}
 We give an argument, originally due to Biswas, ~\cite{B}.

The map $f(E) = (E, 0)$ (i.e. setting $\Phi =0$) takes  $ {\mathcal M}_{1} \rightarrow {\mathcal M}_H$, where ${\mathcal M}_1$ is the moduli space of  stable vector bundles. It in fact  induces an isomorphism $f^*: H^2({\mathcal M}_H, {\mathbb Z}) \rightarrow H^2({\mathcal M}_1, {\mathbb Z})$. It is easy to see this. 

The function $f$ can be thought of the composition of ${\mathcal M}_1 \rightarrow {\mathcal T^*M}_1 \rightarrow {\mathcal M}_H$. The inclusion of ${\mathcal M}_1$ in ${\mathcal T^*M}$ given by $E \rightarrow (E, 0)$ can be easily seen to give isomorphism in homology (since the fibers are contractible). The last map induces a isomorphism at the $H^2$ cohomology level because 
$codim({\mathcal M}_H - T^* {\mathcal M}_1)  \geq 2$. The codimension estimate is given in ~\cite{H2}.
    Given the codimension estimate, it easily follows, by Mayer-Vietoris sequence that the map from  ${\mathcal T^*M}_1 \rightarrow {\mathcal M}_H$ induces isomorphism at $H^2$ level.  (This is because of a general fact that if S is a complex analytic subspace of  a complex manifold $X$
such that the complex codimension of $S$ is at least two, then
the inclusion $X-S  \rightarrow  X$ induces an isomorphism of $H^2$. This
follows from a Mayer-Vietoris sequence). 

The  Kahler form which is the descendent of $\tilde{\Omega}$ on the Hitchn moduli space,  gets mapped  to the usual Kahler form on 
${\mathcal M}_1$ (namely,  the one corresponding to  $ -\frac{1}{2} \int_{\Sigma} {\rm Tr} (\alpha \wedge \beta)$).  This is well known to be integral since there is a determinant bundle on the moduli space ${\mathcal M}_1$ whose curvature is proportional to this Kahler form. This is apparent in the works of Beauville, Drezet, Narasimhan  and others.  Since $f^*$ is an isomorphism, it follows that the Kahler form corresponding to $\tilde{\Omega}$ 
on ${\mathcal M}_H$ is also integral.
\end{proof}
\subsection{The Quantum bundle}

 It can be shown by usual arguments as in  (~\cite{D}, ~\cite{D2}, ~\cite{D3})
${\mathcal P} $ descends to the moduli space of the Hitchin system. Its   curvature is proportional to the  Kahler form which is the descendent of $\tilde{\Omega}$, the constant of proportionality being $\frac{i}{2 \pi} $. Thus it is a prequantum bundle. In fact one can show that ${\mathcal P}$ is  holomorphic with respect to its first holomorphic structure ${\mathcal I}$, ~\cite{D}. Thus we can take square integrable  holomorphic sections as the Hilbert space of quantization. 

{\bf Remarks:}

1. The integerality question is interesting by itself. For instance for the vortex moduli space, the standard symplectic form is integral if the Riemann surface has integral volume. This has been showed by Manton and Nasir when they  computed the volume of the vortex  moduli space ~\cite{MN}.

2. The author came across this method of modifying the Quillen metric for the first time  in ~\cite{BR} where the authors used the algebro-geometric definition of the Quillen bundle in the context of stable triples.

3. Quantization of the other two symplectic forms (which are exact as shown for example in ~\cite{D}) is also possible, where again certain Quillen bundles play the role of the prequantum bundle. This has been showed in ~\cite{D}.

{\bf Acknowledgement:} The author would like to thank Professor Leon Takhtajan for introducing her to Quillen bundles in the first place and Professor Indranil Biswas for the very useful discussions.

\end{document}